\documentclass{amsart}
\usepackage{amsmath} % Advanced math typesetting
\usepackage{amsfonts} % Advanced math fonts
\usepackage{amssymb} % Advanced math symbols
\usepackage[utf8]{inputenc} % Unicode support (Umlauts etc.)
\usepackage[british]{babel} % Change hyphenation rules
\usepackage[protrusion=true,expansion=false]{microtype} % better kerning
\usepackage{mathtools}
\usepackage[T1]{fontenc}
\usepackage{amsthm} % allows theorem and environment customisation
\usepackage{dsfont}
\usepackage{mathrsfs}
\usepackage{xcolor}
\usepackage{lmodern} % latin modern font
\usepackage{hyperref} % Add a link to your document

\definecolor{dblue}{rgb}{0,0,0.70}
\definecolor{dgreen}{rgb}{0,0.60,0}
\hypersetup{
	unicode=true,
	colorlinks=true,
	citecolor=dgreen,
	linkcolor=dblue,
	anchorcolor=dblue
}

\newcommand{\AC}{\mathsf{AC}}
\newcommand{\DC}{\mathsf{DC}}
\newcommand{\GCH}{\mathsf{GCH}}
\newcommand{\HS}{\mathsf{HS}}
\newcommand{\SVC}{\mathsf{SVC}}
\newcommand{\WO}{\mathsf{WO}}
\newcommand{\ZF}{\mathsf{ZF}}
\newcommand{\ZFC}{\mathsf{ZFC}}

\DeclareMathOperator{\Add}{Add}
\DeclareMathOperator{\Aut}{Aut}
\DeclareMathOperator{\dom}{dom}
\DeclareMathOperator{\id}{id}
\DeclareMathOperator{\Ord}{Ord}
\DeclareMathOperator{\Proj}{Proj}
\DeclareMathOperator{\supp}{supp}
\DeclareMathOperator{\sym}{sym}

\newcommand{\bbP}{\mathbb{P}}
\newcommand{\Q}{\mathbb{Q}}

\newcommand{\calM}{\mathcal{M}}
\newcommand{\calN}{\mathcal{N}}

\newcommand{\sF}{\mathscr{F}}
\newcommand{\sG}{\mathscr{G}}
\newcommand{\power}{\mathscr{P}}

\newcommand{\dde}{\dot{e}}
\newcommand{\ddf}{\dot{f}}
\newcommand{\ddm}{\dot{m}}
\newcommand{\ddx}{\dot{x}}
\newcommand{\ddy}{\dot{y}}
\newcommand{\ddV}{\dot{V}}
\newcommand{\ddX}{\dot{X}}

\newcommand{\PGF}{\left\langle\mathbb{P},\mathscr{G},\mathscr{F}\right\rangle}
\newcommand{\ab}{{\alpha,\beta}}
\newcommand{\gab}{{\left\langle\alpha,\beta\right\rangle}}
\newcommand{\aba}[1]{{\alpha#1,\beta#1}}
\newcommand{\gaba}[1]{{\left\langle\alpha#1,\beta#1\right\rangle}}
\newcommand{\abg}{{\alpha,\beta,\gamma}}
\newcommand{\gabg}{{\left\langle\alpha,\beta,\gamma\right\rangle}}

\newcommand{\gabga}[1]{{\left\langle\alpha#1,\beta#1,\gamma#1\right\rangle}}

\newcommand{\abs}[1]{\left|#1\right|}
\newcommand{\gen}[1]{\left\langle#1\right\rangle}
\newcommand{\tup}[1]{\langle#1\rangle}

\newcommand{\comp}{\mathrel{\|}}
\newcommand{\forces}{\mathrel{\Vdash}}
\newcommand{\res}{\nobreak\mskip2mu\mathpunct{}\nonscript
  \mkern-\thinmuskip{\upharpoonright}\mskip6muplus1mu\relax} % copied from the definition of \colon from amssymb

\renewcommand{\leq}{\leqslant}
\renewcommand{\nleq}{\nleqslant}
\renewcommand{\geq}{\geqslant}

\newcommand{\1}{\mathds{1}}
\newcommand*{\defeq}{\mathrel{\vcenter{\baselineskip0.5ex \lineskiplimit0pt
                     \hbox{\scriptsize.}\hbox{\scriptsize.}}}%
                     =}
\newcommand{\vphi}{\varphi}

\newtheoremstyle{boldrk}
  {\topsep}{\topsep}
  {}{}
  {\bfseries}{.}
  {5pt plus 1pt minus 1pt}{}

\theoremstyle{plain}
\newtheorem{thm}{Theorem}[section]
\newtheorem*{mainthm}{Main Theorem}
\newtheorem{lem}[thm]{Lemma}
\newtheorem{prop}[thm]{Proposition}

\newtheorem{claim}{Claim}[thm] % takes a subordinate numbering to whatever proposition the environment is hiding in

\theoremstyle{boldrk}
\newtheorem{qn}[thm]{Question}

\theoremstyle{definition}
\newtheorem{defn}[thm]{Definition}

\newenvironment*{poc}[1][Proof of Claim]{\begin{proof}[#1]}{\end{proof}} % Produces a proof-like environment with an alternativ qed symbol

\title[Hartogs and Lindenbaum]{Which pairs of cardinals can be\linebreak Hartogs and Lindenbaum numbers of a set?}
\author{Asaf Karagila}
\author{Calliope Ryan-Smith}

\email{karagila@math.huji.ac.il}
\urladdr{https://karagila.org}
\email{c.Ryan-Smith@leeds.ac.uk}
\urladdr{https://academic.calliope.mx}
\address{School of Mathematics, University of Leeds, LS2 9JT, UK}

\date{\today}
\keywords{Symmetric extensions, iterated forcing, iterated symmetries, symmetric iteration, axiom of choice, Hartogs number, Lindenbaum number.}
\thanks{
	No data are associated with this article.
	The first author was supported by a UKRI Future Leaders Fellowship [MR/T021705/2].
	The second author's work was financially supported by EPSRC via the Mathematical Sciences Doctoral Training Partnership [EP/W523860/1].
	For the purpose of open access, the author has applied a Creative Commons Attribution (CC-BY) licence to any Author Accepted Manuscript version arising from this submission.
}
\subjclass[2020]{Primary: 03E25; Secondary 03E35}

\begin{document}

\begin{abstract}
Given any $\lambda\leq\kappa$, we construct a symmetric extension in which there is a set $X$ such that $\aleph(X)=\lambda$ and $\aleph^*(X)=\kappa$. Consequently, we show that $\ZF+$``For all pairs of infinite cardinals $\lambda\leq\kappa$ there is a set $X$ such that $\aleph(X)=\lambda\leq\kappa=\aleph^*(X)$'' is consistent.
\end{abstract}

\maketitle

\section{Introduction}

The Axiom of Choice is one of the most successful axioms in modern mathematics, generating both many applications as well as several ``paradoxes'' (or rather counterintuitive surprises). One of its famous equivalents is Zermelo's theorem stating that every set can be well-ordered. Therefore, if the Axiom of Choice fails in a universe of set theory, some sets cannot be well-ordered. Nevertheless, we can still consider two ways in which a set $X$ is ``large'', by asking how large are well-orderable subsets of $X$ and how large are well-orderable partitions of $X$.

For example, if a set $X$ can be mapped onto $\omega_{13}$, then at the very least there is a sense in which it is large compared to $\omega$. Moreover, if we extend the universe so that $X$ can be well-ordered, then (at least) one of two scenarios must hold: (1) $X$ will have cardinality of at least $\aleph_{13}$; or (2) $\omega_{13}$ will be collapsed.

\begin{defn}
Let $X$ be a set. The \emph{Hartogs number} of $X$ is
\begin{equation*}
\aleph(X)\defeq\min\{\alpha\in\Ord\mid\text{ There is no injection }f\colon\alpha\to X\}.
\end{equation*}
The \emph{Lindenbaum number} of $X$ is
\begin{equation*}
\aleph^*(X)\defeq\min\{\alpha\in\Ord\backslash\{0\}\mid\text{There is no surjection }f\colon X\to\alpha\}.
\end{equation*}
\end{defn}
The existence of $\aleph(X)$ is guaranteed under $\ZF$ by Hartogs's lemma, from \cite{hartogs_uber_1915}.\footnote{One can also go to \cite[Theorem~8.18]{goldrei_classic_1998} for a proof in English.} Furthermore, $\aleph(X)$ must also be a cardinal number, and when $X$ is well-orderable, $\aleph(X)=\abs{X}^+$. The existence of $\aleph^*(X)$ is guaranteed under $\ZF$ by a lemma first used in the proof of Lindenbaum's theorem.\footnote{The theorem was first stated without proof in \cite[Th\'eor\`eme~82.$A_6$]{lindenbaum_communication_1926}. The first published proof is in \cite{sierpinski_sur-lindenbaum_1947}, or can be found in English in \cite[Ch.~XVI, Section~$3$, Theorem~1]{sierpinski_cardinal_1965}.} Again, $\aleph^*(X)$ must also be a cardinal number, and when $X$ is well-orderable, $\aleph^*(X)=\abs{X}^+$.

It is also not difficult to see that for any set $X$, $\aleph(X)\leq\aleph^*(X)$. However, it need not be the case that $\aleph=\aleph^*$. Indeed, in $\ZF$ the statement $(\forall X)\aleph(X)=\aleph^*(X)$ is equivalent to Axiom of Choice for well-ordered families of sets, established in \cite{pelc_on_1978} (this axiom is weaker than the Axiom of Choice in its general form). So, if the Axiom of Choice for well-ordered families fails, there is some $X$ such that $\aleph(X)<\aleph^*(X)$. We are concerned with a maximal possible violation of this principle.

\begin{mainthm}\label{thm:main-theorem}
$\ZF$ is equiconsistent with $\ZF+$``for all infinite cardinals $\lambda\leq\kappa$ there is a set $X$ such that $\aleph(X)=\lambda\leq\kappa=\aleph^*(X)$''.
\end{mainthm}

\subsection{Structure of the paper}
Section~\ref{s:preliminaries} establishes preliminaries for the paper, particularly our conventions for handling cardinalities, forcing, and symmetric extensions. Some time is also given to permutation groups. In Section~\ref{s:one-step} we show the consistency of the existence of a single set $X$ such that $\aleph(X)=\lambda$ and $\aleph^*(X)=\kappa$ for arbitrary infinite $\lambda\leq\kappa$. In Section~\ref{s:iteration} we use the machinery established in Section~\ref{s:one-step} to construct a class-sized product of notions of forcing to prove the main theorem.

\section{Preliminaries}\label{s:preliminaries}
\noindent Throughout this paper we work in $\ZFC$. Our treatment of forcing will be standard. By a \emph{notion of forcing} we mean a preordered set $\bbP$ with maximum element denoted $\1_\bbP$, or with the subscript omitted when clear from context. We write $q\leq p$ to mean that $q$ \emph{extends} $p$. Two conditions $p,p'$ are said to be \emph{compatible}, written $p\comp p'$, if they have a common extension. We follow Goldstern's alphabet convention so $p$ is never a stronger condition than $q$, etc.

When given a collection of $\bbP$-names, $\{\ddx_i\mid i\in I\}$, we will denote by $\{\ddx_i\mid i\in I\}^\bullet$ the canonical name this class generates: $\{\tup{\1,\ddx_i}\mid i\in I\}$. The notation extends naturally to ordered pairs and functions with domains in the ground model. An immediate application of this is a simplified definition of check names; given $x$, the check name for $x$ is defined inductively as $\check{x}=\{\check{y}\mid y\in x\}^\bullet$.

Given a set $X$, we denote by $\abs{X}$ its cardinal number. If $X$ can be well-ordered, then $\abs{X}$ is simply the least ordinal $\alpha$ such that a bijection between $\alpha$ and $X$ exists. Otherwise, we use the Scott cardinal of $X$, ${\{Y\in V_\alpha\mid\exists f\colon X\to Y\text{ a bijection}\}}$ with $\alpha$ taken minimal such that the set is non-empty. Greek letters, when used as cardinals, always refer to well-ordered cardinals. We call an ordinal $\alpha$ a cardinal if $\abs{\alpha}=\alpha$.

We write $\abs{X}\leq\abs{Y}$ to mean that there is an injection from $X$ to $Y$, and ${\abs{X}\leq^*\abs{Y}}$ to mean that there is a surjection from $Y$ to $X$ or that $X$ is empty. These notations extend to $\abs{X}<\abs{Y}$ (and $\abs{X}<^*\abs{Y}$) to mean that $\abs{X}\leq\abs{Y}$ (respectively $\abs{X}\leq^*\abs{Y}$) and there is no injection from $Y$ to $X$ (respectively no surjection from $X$ to $Y$). Finally, $\abs{X}=\abs{Y}$ means that there is a bijection between $X$ and $Y$.

Using this notation, one may redefine the Hartogs and Lindenbaum numbers as
\begin{align*}
\aleph(X)&\defeq\min\{\alpha\in\Ord\mid\abs{\alpha}\nleq\abs{X}\}\text{ and}\\
\aleph^*(X)&\defeq\min\{\alpha\in\Ord\mid\abs{\alpha}\nleq^*\abs{X}\}.
\end{align*}

\subsection{Symmetric extensions}
It is key to the role of forcing that if $V\vDash\ZFC$, and $G$ is $V$-generic for some notion of forcing $\bbP\in V$, then $V[G]\vDash\ZFC$. However, this demands additional techniques for trying to establish results that are inconsistent with $\AC$. Symmetric extensions expand the technique of forcing in this very way by constructing an intermediate model between $V$ and $V[G]$ that is a model of $\ZF$.

Given a notion of forcing $\bbP$, we shall denote by $\Aut(\bbP)$ the collection of automorphisms of $\bbP$. Let $\bbP$ be a notion of forcing and $\pi\in\Aut(\bbP)$. Then $\pi$ extends naturally to act on $\bbP$-names by recursion: $\pi\ddx=\{\tup{\pi p,\pi\ddy}\mid\tup{p,\ddy}\in\ddx\}$.

Due to the construction of the forcing relation from the notion of forcing, we end up with the following lemma, proved in \cite[Lemma~14.37]{jech_set_2003}.
\begin{lem}[The Symmetry Lemma]
Let $\bbP$ be a notion of forcing, $\pi\in\Aut(\bbP)$, and $\ddx$ a $\bbP$-name. Then $p\forces\vphi(\ddx)$ if and only if $\pi p\forces\vphi(\pi\ddx)$.\qed
\end{lem}
\noindent Note in particular that for all $\pi\in\Aut(\bbP)$ we have $\pi\1=\1$. Therefore, $\pi\check{x}=\check{x}$ for all ground model sets $x$, and $\pi\{\ddx_i\mid i\in I\}^\bullet=\{\pi\ddx_i\mid i\in I\}^\bullet$, similarly extending to tuples, functions, etc.

Given a group $\sG$, a \emph{filter of subgroups} of $\sG$ is a set $\sF$ of subgroups of $\sG$ that is closed under supergroups and finite intersections. We say that $\sF$ is \emph{normal} if whenever $H\in\sF$ and $\pi\in\sG$, then $\pi H\pi^{-1}\in\sF$.

A \emph{symmetric system} is a triple $\PGF$ such that $\bbP$ is a notion of forcing, $\sG$ is a group of automorphisms of $\bbP$, and $\sF$ is a normal filter of subgroups of $\sG$. Given such a symmetric system, we say that a $\bbP$-name $\ddx$ is \emph{$\sF$-symmetric} if $\sym_\sG(\ddx)=\{\pi\in\sG\mid\pi\ddx=\ddx\}\in\sF$. $\ddx$ is \emph{hereditarily $\sF$-symmetric} if this notion holds for every $\bbP$-name hereditarily appearing in $\ddx$. We denote by $\HS_\sF$ the class of hereditarily $\sF$-symmetric names. When clear from context, we will omit subscripts and simply write $\sym(\ddx)$ or $\HS$. The following theorem, \cite[Lemma~15.51]{jech_set_2003}, is foundational to the study of symmetric extensions.
\begin{thm}
Let $\PGF$ be a symmetric system, $G\subseteq\bbP$ a $V$-generic filter, and let $\calM$ denote the class $\HS^G_\sF=\{\ddx^G\mid\ddx\in\HS_\sF\}$. Then $\calM$ is a transitive model of $\ZF$ such that $V\subseteq\calM\subseteq V[G]$.\hfil\qed
\end{thm}
Finally, we have a forcing relation for symmetric extensions $\forces^\HS$ defined by relativising the forcing relation $\forces$ to the class $\HS$. This relation has the same properties and behaviour of the standard forcing relation $\forces$. Moreover, when $\pi\in\sG$, the Symmetry Lemma holds for $\forces^\HS$.

\subsection{Wreath products}
Frequently within this paper we will exhibit groups of automorphisms $\sG$ as permutation groups with an action on the notion of forcing. By a \emph{permutation group} (of the set $X$) we mean a subgroup of $S_X$, the group of bijections $X\to X$. If $\pi\in S_X$, then by the \emph{support} of $\pi$, written $\supp(\pi)$, we mean the set $\{x\in X\mid\pi(x)\neq x\}$. Given an infinite cardinal $\lambda$ we denote by $S_X^{{<}\lambda}$ the subgroup of $S_X$ of permutations $\pi$ such that $\abs{\supp(\pi)}<\lambda$.
\begin{defn}[Wreath product]
Given two permutation groups $G\leq S_X$ and $H\leq S_Y$, the \emph{wreath product} of $G$ and $H$, denoted $G\wr H$, is the subgroup of permutations $\pi\in S_{X\times Y}$ which have the following property:
\begin{quote}
There is $\pi^*\in G$ and a sequence $\gen{\pi_x\mid x\in X}\in H^X$ such that for all $\tup{x,y}\in X\times Y$, $\pi(x,y)=\tup{\pi^*(x),\pi_x(y)}$.
\end{quote}
That is, $\pi$ first permutes each column $\{x\}\times Y$ according to some $\pi_x\in H$, and then acts on the $X$ co-ordinate of $X\times Y$, permuting its columns via some $\pi^*\in G$.
\end{defn}
Given $\pi\in G\wr H$, we will use the notation $\pi^*$ and $\pi_x$ to mean the elements of $G$ and $H$ respectively from the definition. Note that if $\pi,\sigma\in G\wr H$, then $(\pi\sigma)^*=\pi^*\sigma^*$.

Note also that $\{\id\}\wr S_Y\leq S_{X\times Y}$ is the group of all $\pi\in S_{X\times Y}$ such that for all ${\tup{x,y}\in X\times Y}$, $\pi(x,y)\in\{x\}\times Y$.

\section{Realising a single pair as Hartogs and Lindenbaum of a set}\label{s:one-step}
\noindent Let us spend some time establishing the consistency and construction of a single set $X$ such that $\aleph(X)=\lambda$ and $\aleph^*(X)=\kappa$. The construction used here will then be iterated in Section~\ref{s:iteration} to prove our main theorem.

\begin{thm}\label{thm:one-step}
Let $\lambda\leq\kappa$ be infinite cardinals. There is a symmetric system $\PGF$ and a $\bbP$-name $\ddX\in\HS_\sF$ such that
\begin{equation*}
\1_\bbP\forces^\HS``\aleph(\ddX)=\check{\lambda}\text{ and }\aleph^*(\ddX)=\check{\kappa}".
\end{equation*}
\end{thm}

\begin{proof}
Let $\mu$ be a regular cardinal such that $\mu\geq\lambda$, and let $\bbP=\Add(\mu,\kappa\times\lambda\times\mu)$. That is, the conditions of $\bbP$ are partial functions $p\colon\kappa\times\lambda\times\mu\times\mu\to2$ such that $\abs{\dom(p)}<\mu$, with $q\leq p$ if $q\supseteq p$.

For $p\in\bbP$ and $A\subseteq\kappa\times\lambda\times\mu$, we will write $p\res A$ to mean $p\res A\times\mu$, and for $B\subseteq\kappa\times\lambda$ we will write $p\res B$ to mean $p\res B\times\mu\times\mu$. Furthermore, we shall write $\supp(p)$ to mean the projection of the domain of $p$ to its first three co-ordinates, so $\supp(p)\subseteq\kappa\times\lambda\times\mu$.

We define the following $\bbP$-names:
\begin{enumerate}
\item $\ddy_{\abg}\defeq\{\tup{p,\check{\delta}}\mid p\in\bbP,\delta<\mu,p(\abg,\delta)=1\}$;
\item $\ddx_{\ab}\defeq\{\ddy_{\abg}\mid\gamma\in\mu\}^\bullet$; and
\item $\ddX\defeq\{\ddx_{\ab}\mid\gab\in\kappa\times\lambda\}^\bullet$.
\end{enumerate}
In the extension, $\ddX$ will be the name for the set $X$ such that $\aleph(X)=\lambda$ and $\aleph^*(X)=\kappa$.

Let $\sG=S_{\kappa\times\lambda}^{{<}\lambda}\wr S_{\mu}$. That is to say, $\sG$ is the group of permutations $\pi$ in the wreath product $S_{\kappa\times\lambda}\wr S_\mu$ such that $\pi^*\in S_{\kappa\times\lambda}$ fixes all but fewer than $\lambda$-many elements of $\kappa\times\lambda$. $\sG$ acts on $\bbP$ via $\pi p(\pi(\abg),\delta)=p(\abg,\delta)$. Note that, for $\pi\in\sG$,
\begin{align*}
\pi\ddy_{\abg}&=\{\tup{\pi p,\pi\check{\delta}}\mid p\in\bbP,\delta<\mu,p(\abg,\delta)=1\}\\
&=\{\tup{\pi p,\check{\delta}}\mid p\in\bbP,\delta<\mu,\pi p(\pi(\abg),\delta)=1\}\\
&=\{\tup{p,\check{\delta}}\mid p\in\bbP,\delta<\mu,p(\pi(\abg),\delta)=1\}\\
&=\ddy_{\pi(\abg)}.
\end{align*}
Similar verification shows that $\pi\ddx_{\ab}=\ddx_{\pi^*(\ab)}$ and $\pi\ddX=\ddX$. When we have defined the filter of subgroups $\sF$ (which we shall do upon the conclusion of this sentence), it will be clear from these calculations that these names are hereditarily $\sF$-symmetric.

For $I\in[\kappa]^{{<}\kappa}$, $J\in[I\times\lambda]^{{<}\lambda}$, and $K\in[J\times\mu]^{{<}\lambda}$, let $H_{I,J,K}$ be the subgroup of $\sG$ given by those $\pi$ such that:
\begin{enumerate}
\item $\pi^*\res I\times\lambda\in\{\id\}\wr S_\lambda$;
\item $\pi^*\res J=\id$; and
\item $\pi\res K=\id$.
\end{enumerate}
That is, we are taking those $\pi\in S_{\kappa\times\lambda}^{{<}\lambda}\wr S_\mu$ such that $\pi^*$ fixes the columns taken from the set $I$ of cardinality less than $\kappa$, and fixes pointwise the set $J$ of cardinality less than $\lambda$. We then further require that $\pi$ fixes pointwise the set $K$ of cardinality less than $\lambda$.

Let $\sF$ be the filter of subgroups of $\sG$ generated by groups of the form $H_{I,J,K}$ for $I\in[\kappa]^{{<}\kappa}$, $J\in[I\times\lambda]^{{<}\lambda}$, and $K\in[J\times\mu]^{{<}\lambda}$.\footnote{Since $\bbP$ is $\lambda$-closed and $\sF$ is $\lambda$-complete, $\DC_{<\lambda}$ holds in the symmetric extension. A proof can be found in \cite[Lemma~1]{karagila_preserving_dc_2019}.} We shall refer to triples $I,J,K$ as being `appropriate' to mean that they satisfy these conditions.

By our previous calculations, $\pi\ddy_{\abg}=\ddy_{\abg}$ whenever $\pi(\abg)=\gabg$, so $\sym(\ddy_{\abg})\geq H_{\{\alpha\},\{\gab\},\{\gabg\}}\in\sF$. Similarly $\sym(\ddx_{\ab})\geq H_{\{\alpha\},\{\ab\},\emptyset}\in\sF$ and of course $\sym(\ddX)=\sG\in\sF$.

\begin{claim}
$\sF$ is normal. Hence $\gen{\bbP,\sG,\sF}$ is a symmetric system.
\end{claim}

\begin{poc}
Note that for appropriate $I,J,K$ and $I',J',K'$,
\begin{equation*}
H_{I,J,K}\cap H_{I',J',K'}=H_{I\cup I',J\cup J',K\cup K'}
\end{equation*}
and $I\cup I',J\cup J',K\cup K'$ is appropriate. Therefore, for all $H\leq\sG$, $H\in\sF$ if and only if there is appropriate $I,J,K$ such that $H\geq H_{I,J,K}$. Hence, to show that $\sF$ is normal, it is sufficient to show that for all appropriate $I,J,K$ and all $\pi\in\sG$, there is appropriate $I',J',K'$ such that $\pi H_{I,J,K}\pi^{-1}\geq H_{I',J',K'}$, or equivalently that $H_{I,J,K}\geq\pi^{-1}H_{I',J',K'}\pi$. Given such $I,J,K$ and $\pi$, we define
\begin{align*}
K'&=\pi``K\\
J'&=\Proj(K')\cup\supp(\pi^*)\cup J\cup\pi^*``J\\
I'&=\Proj(J')\cup I.
\end{align*}
We must first show that $I',J',K'$ is appropriate. Firstly, note that $\abs{K'}=\abs{K}<\lambda$, $\abs{J'}\leq2\cdot\abs{J}+\abs{\supp(\pi^*)}+\abs{K'}<\lambda$, and $\abs{I'}\leq\abs{I}+\abs{J'}<\kappa$ as required. Secondly, the inclusion of the projections in the definitions of $I',J',K'$ guarantee that they are appropriate. We claim that $H_{I',J',K'}$ is the required group. Let $\sigma\in H_{I',J',K'}$, then we must show that $\pi^{-1}\sigma\pi\in H_{I,J,K}$.

Firstly, for all $\gabg\in K$, $\pi(\abg)\in K'$, so ${\sigma(\pi(\abg))=\pi(\abg)}$ and hence $\pi^{-1}\sigma\pi(\abg)=\gabg$ as required.

Secondly, we claim that $(\pi^{-1}\sigma\pi)^*=\sigma^*$ and that this is sufficient. Indeed if this is the case then, since $J\subseteq J'$ and $I\subseteq I'$, we get that $\pi^{-1}\sigma\pi\in H_{I,J,K}$ as desired. Note also that since $\pi^*$ is a bijection, $\gab\in\supp(\pi^*)$ if and only if ${\pi^*(\ab)\in\supp(\pi^*)}$.

If $\pi^*(\ab)\neq\gab$ then both $\gab$ and $\pi^*(\ab)$ are in $\supp(\pi^*)\subseteq J'$, and thus ${(\sigma\pi)^*(\ab)=\pi^*(\ab)}$ and $\sigma^*(\ab)=\gab$. Hence ${(\pi^{-1}\sigma\pi)^*(\ab)=\gab}$ and ${\gab=\sigma^*(\ab)}$ as desired.

Finally we deal with the case $\pi^*(\ab)=\gab$. If $(\pi\sigma)^*(\ab)\neq\sigma^*(\ab)$, then $\sigma^*(\ab)\in\supp(\pi^*)\subseteq J'$, so ${(\sigma\sigma)^*(\ab)=\sigma^*(\ab)}$ and thus ${\sigma^*(\ab)=\gab}$. Hence, $\pi^*(\ab)\neq\gab$, and as before ${(\pi^{-1}\sigma\pi)^*(\ab)=\sigma^*(\ab)}$. Therefore, if $\pi^*(\ab)=\gab$ then we must have $(\pi\sigma)^*(\ab)=\sigma^*(\ab)$. $\pi^*(\ab)=\gab$, so $(\pi\sigma)^*(\ab)=(\sigma\pi)^*(\ab)$ and ${(\pi^{-1}\sigma\pi)^*(\ab)=\sigma^*(\ab)}$ as desired.
\end{poc}

\begin{claim}\label{claim:homogeneity}
Let $q\in\bbP$, $H=H_{I,J,K}\in\sF$, and $\gab,\gaba'\in\kappa\times\lambda$. The following are equivalent:
\begin{enumerate}
\item There is $\pi\in H$ such that $\pi^*$ is the transposition $\begin{pmatrix}\gab&\gaba'\end{pmatrix}$ and $\pi q\comp q$.
\item $\{\alpha,\alpha'\}\cap I\neq\emptyset\implies\alpha=\alpha'$, and

$\{\gab,\gaba'\}\cap J\neq\emptyset\implies\gab=\gaba'$.
\end{enumerate}
\end{claim}
\begin{poc}
($(1)\implies(2)$). By the definition of $H$, if there is such a $\pi\in H$ then Condition (2) must be satisfied.

($(2)\implies(1)$). If $\alpha,\alpha',\beta,\beta'$ satisfy Condition (2), then any $\pi\in\sG$ such that $\pi^*=\begin{pmatrix}\gab&\gaba'\end{pmatrix}$ is a candidate for an element of $H$ (as $\abs{\supp(\pi^*)}=2<\lambda$). Firstly, if $\gab=\gaba'$ then we may take $\pi=\id$, so assume otherwise. Let $A=\{\gamma\in\mu\mid\gabg\in\supp(q)\}$, and $B=\{\gamma'\in\mu\mid\gabga'\in\supp(q)\}$. Since $\abs{A},\abs{B}<\mu$, there is a permutation $\sigma$ of $\mu$ such that $\sigma``A\cap B=\emptyset$ and $A\cap\sigma``B=\emptyset$. Therefore, setting $\pi_{\ab}=\pi_{\aba'}=\sigma$ we will have that
\begin{align*}
\supp(q\res\gab)\cap\supp(\pi q\res\gab)&=\sigma``A\cap B=\emptyset\\
\intertext{and}
\supp(q\res\gaba')\cap\supp(\pi q\res\gaba')&=A\cap\sigma``B=\emptyset.
\end{align*}
Hence ${q\res\{\gab,\gaba'\}\comp \pi q\res\{\gab,\gaba'\}}$, and for all other $\gaba{''}$ we have ${q\res\gaba{''}}={\pi q\res\gaba{''}}$.
\end{poc}

The remainder of the proof will be spent showing that the name $\ddX$ will give us the object that we are searching for, that is $\1\forces^\HS\aleph(\ddX)=\check{\lambda}$ and $\1\forces^\HS\aleph^*(\ddX)=\check{\kappa}$. We shall first prove the inequalities $\1\forces^\HS\aleph(\ddX)\geq\check{\lambda}$ and $\1\forces^\HS\aleph^*(\ddX)\geq\check{\kappa}$, and then prove that they can be sharpened to equalities.

Towards the inequalities, for any $\alpha,\eta<\kappa$ let
\begin{equation*}
\iota_{\alpha,\eta}\defeq\begin{cases}
\alpha&\alpha<\eta\\
0&\text{Otherwise.}
\end{cases}
\end{equation*}
Then consider the name $\dde_\eta\defeq\{\tup{\ddx_{\ab},\check{\iota}_{\alpha,\eta}}^\bullet\mid\gab\in\kappa\times\lambda\}^\bullet$. Routine verification shows $\sym(\dde_\eta)\geq H_{\eta,\emptyset,\emptyset}$, so $\dde_\eta\in\HS$. Furthermore, ${\1\forces``\dde_\eta\colon\ddX\to\check{\eta}\text{ is surjective}"}$, and thus $\1\forces^\HS\aleph^*(\ddX)\geq\check{\kappa}$.

Similarly, for any $\eta<\lambda$ and any $\alpha\in\kappa$, take $\ddm_\eta\defeq\{\tup{\check{\beta},\ddx_{\ab}}^\bullet\mid\beta<\eta\}^\bullet$.  Routine verification shows that $\sym(\ddm_\eta)\geq H_{\{\alpha\},\{\alpha\}\times\eta,\emptyset}$, so $\ddm_\eta\in\HS$ as well. Furthermore, $\1\forces``\ddm_\eta\colon\check{\eta}\to\ddX$ is injective'', and thus $\1\forces^\HS\aleph(\ddX)\geq\check{\lambda}$.

It remains to show that these inequalities are, in fact, equalities, starting with $\aleph^*$. Suppose that $\ddf\in\HS$ and $p\forces\ddf\colon\ddX\to\check{\kappa}$. Let $H=H_{I,J,K}\leq\sym(\ddf)$. Then suppose that for some $q\leq p$ and $\gab\in\kappa\times\lambda$ there is $\eta$ such that $q\forces\ddf(\ddx_{\ab})=\check{\eta}$.

By Claim~\ref{claim:homogeneity}, if $\alpha\notin I$ then for any $\alpha'\notin I$ and any $\beta'\in\lambda$ there is $\pi\in H$ such that $\pi^*(\ab)=\gaba'$ and $\pi q\comp q$. Then $\pi q\forces\ddf(\ddx_{\aba'})=\check{\eta}$, so $q\cup\pi q\leq q$ forces that $\ddf(\ddx_{\ab})=\ddf(\ddx_{\aba'})$. Hence $p$ forces that $\ddf$ is constant outside of $I\times\lambda$.

If instead $\alpha\in I$ but $\gab\notin J$ then again by Claim~\ref{claim:homogeneity}, for any $\beta'\in\lambda$ such that $\tup{\alpha,\beta'}\notin J$ there is $\pi\in H$ such that $\pi^*(\ab)=\tup{\alpha,\beta'}$ and $\pi q\comp q$. Once again $\pi q\forces\ddf(\ddx_{\aba'})=\check{\eta}$, and so $p$ forces that in $(I\times\lambda)\backslash J$, the value of $\ddf(\ddx_{\ab})$ depends only on $\alpha$. This means that $\ddf$ can take only at most $\abs{J}+\abs{I}+1<\kappa$ many values, so cannot be a surjection, and thus $\1\forces^\HS\aleph^*(\ddX)=\check{\kappa}$.

Finally, suppose that $\ddf\in\HS$ and $p\forces\ddf\colon\check{\lambda}\to\ddX$. Let $H=H_{I,J,K}\leq\sym(\ddf)$. We shall show that $p\forces\ddf``\check{\lambda}\subseteq\{\ddx_{\ab}\mid\gab\in J\}^\bullet$, and hence $\ddf$ cannot be injective.

Suppose otherwise, that for some $q\leq p$, $\gab\notin J$, and $\eta<\lambda$ we have that ${q\forces\ddf(\check{\eta})=\ddx_{\ab}}$. Since $\gab\notin J$, for any $\beta'\in\lambda$ such that $\tup{\alpha,\beta'}\notin J$ there is ${\pi\in H}$ such that $\pi^*(\ab)=\tup{\alpha,\beta'}$ and $\pi q\comp q$. Since $\abs{J}<\lambda$ we may take $\beta'\neq\beta$, and so $\pi q\forces\ddf(\check{\eta})=\pi\ddx_{\ab}=\ddx_{\alpha,\beta'}$. Therefore, ${\pi q\cup q\forces\ddx_{\ab}=\ddf(\check{\eta})=\ddx_{\alpha,\beta'}}$, contradicting our assumption that $\beta'\neq\beta$. Thus our assertion is proved and ${\1\forces^\HS\aleph(\ddX)=\check{\lambda}}$.
\end{proof}

\section{Realising all pairs at once}\label{s:iteration}
\noindent We have now constructed enough technology to prove our main theorem.

\begin{mainthm}\label{thm:main-thm}
$\ZF$ is equiconsistent with $\ZF+``$For all pairs of infinite cardinals $\lambda\leq\kappa$, there is a set $X$ such that $\aleph(X)=\lambda$ and $\aleph^*(X)=\kappa$''.
\end{mainthm}

The structure shall be similar to the treatment of class products of symmetric extensions found in, for example, \cite{karagila_fodors_2018}.

We shall begin in a model $V$ of $\ZFC+\GCH$. We shall define inductively a symmetric system $\tup{\bbP_\alpha,\sG_\alpha,\sF_\alpha}$ for each $\alpha\in\Ord$. Each such system will be precisely of the form described in Theorem~\ref{thm:one-step}, and so to fully define each system we need only define the parameters $\lambda_\alpha,\kappa_\alpha$, and $\mu_\alpha$. First, let $\{\tup{\lambda_\alpha,\kappa_\alpha}\mid\alpha\in\Ord\}$ be an enumeration of each pair $\tup{\lambda,\kappa}$ with $\aleph_0\leq\lambda\leq\kappa$, using (for example) the G\"odel pairing function. Then we shall define $\mu_\alpha$ to be the least cardinal satisfying the following conditions:
\begin{enumerate}
\item $\mu_\alpha$ is regular;
\item for all $\beta<\alpha$, $\mu_\beta<\mu_\alpha$;
\item for all $\beta\leq\alpha$, $\kappa_\beta<\mu_\alpha$;
\item setting $\Q$ to be the finite-support product $\prod_{\beta<\alpha}\bbP_\beta$, we require $\abs{\Q}<\mu_\alpha$;
\item for all $\beta<\alpha$, $\1_\Q\forces_{\Q}|\ddV_{\mu_\beta^+}|<\check{\mu}_\alpha$; and
\item $\aleph_\alpha<\mu_\alpha$.
\end{enumerate}
Let $\bbP$ be the finite-support product of all $\bbP_\alpha$, $\sG$ the finite-support product of all $\sG_\alpha$, and $\sF$ the finite-support product of all $\sF_\alpha$. For $E\subseteq\Ord$, we denote by $\bbP\res E$ (respectively $\sG\res E,\sF\res E$) the restriction of $\bbP$ (respectively $\sG,\sF$) to the co-ordinates found in $E$. Since any $\bbP$-name $\ddx$ is a set, it is a $\bbP\res\alpha$-name for some $\alpha$, and so $\ddx$ is hereditarily $\sF$-symmetric if and only if it is hereditarily $\sF\res\alpha$-symmetric for some $\alpha$. Therefore, setting $\HS=\HS_\sF$, $\HS_\alpha=\HS_{\sF\res\alpha}$, and letting $G$ be $V$-generic for $\bbP$, we get that
\begin{equation*}
\bigcup_{\alpha\in\Ord}\HS_\alpha^{G\res\alpha}=\bigcup_{\alpha\in\Ord}\HS_\alpha^G=\left(\bigcup_{\alpha\in\Ord}\HS_\alpha\right)^G=\HS^G.
\end{equation*}
Let $\calM=\HS^G$ and $\calM_\alpha=\HS_\alpha^{G\res\alpha}$. Then we have that $\calM=\bigcup_{\alpha\in\Ord}\calM_\alpha$. We wish to prove that $\calM\vDash\ZF$, and shall use the following theorem, \cite[Theorem~9.2]{karagila_iterating_2019}.

\begin{thm}\label{thm:limit-rank}
Let $\tup{\bbP_\alpha,\sG_\alpha,\sF_\alpha\mid\alpha\in\Ord}$ be a finite-support product of symmetric extensions of homogeneous systems. Suppose that for each $\eta$ there is $\alpha^*$ such that for all $\alpha\geq\alpha^*$, the $\alpha$th symmetric extension does not add new sets of rank at most $\eta$. Then no sets of rank at most $\eta$ are added by limit steps either. In particular, the end model satisfies $\ZF$.\qed
\end{thm}

The conditions of the theorem are also desirable for our construction. We shall show that for all $\alpha$ there is a hereditarily symmetric name $\ddX_\alpha$ such that $\ddX_\alpha\in\HS_{\alpha+1}$ and $\calM_{\alpha+1}\vDash``\aleph(\ddX_\alpha^G)=\lambda_\alpha\text{ and }\aleph^*(\ddX_\alpha^G)=\kappa_\alpha"$. In this case, if we can preserve a large enough initial segment of $\calM_{\alpha+1}$ for the rest of the iteration, then $\ddX^G_\alpha$ will still have this property in $\calM$.

We shall require the following fact, that is proved in \cite[Lemma~2.3]{karagila_fodors_2018}.

\begin{lem}\label{lem:the-weird-round-one}
Let $\kappa$ be a regular cardinal, $\bbP$ a $\kappa$-c.c. forcing, and $\Q$ a $\kappa\text{-distributive}$ forcing. If $\1_\Q\forces``\check{\bbP}$ is $\check{\kappa}$-c.c.'', then $\1_\bbP\forces``\check{\Q}$ is $\check{\kappa}$-distributive''.\qed
\end{lem}

\begin{prop}\label{prop:prime-directive}
Let $\delta<\beta<\alpha$. Then $\calM_\beta$ and $\calM_\alpha$ agree on sets of rank less than $\mu_\delta^+$.
\end{prop}

\begin{proof}
It is sufficient to prove that for all $\beta\in\Ord$, $\calM_\beta$ and $\calM_{\beta+1}$ agree on sets of rank less than $\mu_\delta$; this is the successor stage for an induction on $\alpha$ of the statement of the Proposition, and Theorem~\ref{thm:limit-rank} provides the induction at the limit stage. Towards this end, let $\delta<\beta\in\Ord$ and $\calN=V[G\res\beta]$. We shall show that $\bbP_\beta$ adds no sets of rank less than $\mu_\delta^+$ to $\calN$, and since $\calM_\beta\subseteq\calN$ and $\calM_{\beta+1}\subseteq\calN[G(\beta)]$, the claim is proved.

Let $\kappa=|V_{\mu_\delta^+}^\calN|$. Then, by the definition of $\mu_\beta$, $\kappa<\mu_\beta$, and so it is sufficient to prove that $\bbP_\beta$ is $\mu_\beta$-distributive. $\bbP_\beta=\Add(\mu_\beta,\kappa_\beta\times\lambda_\beta\times\mu_\beta)^V$, and so certainly in $V$ it is $\mu_\beta^+$-distributive. Furthermore, by definition, $\abs{\bbP\res\beta}<\mu_\beta$, so $\bbP\res\beta$ is $\mu_\beta$-c.c., and indeed $\bbP_\beta\forces``\check{\bbP}\res\beta$ is $\check{\mu}_\beta$-c.c.''. Hence, by Lemma~\ref{lem:the-weird-round-one}, ${\bbP\res\beta\forces``\check{\bbP}_\beta\text{ is }\mu_\beta\text{-distributive}"}$, as required.
\end{proof}

\begin{prop}\label{prop:witness-hartlin}
For all $\alpha\in\Ord$, $\calM_{\alpha+1}\vDash(\exists X_\alpha)(\aleph(X_\alpha)=\lambda_\alpha\leq\kappa_\alpha=\aleph^*(X_\alpha))$.
\end{prop}

\begin{proof}
Let $\calN$ be the symmetric extension of $V[G\res\alpha]$ given by the symmetric system $\tup{\bbP_\alpha,\sG_\alpha,\sF_\alpha}$ and the generic filter $G(\alpha)$. By the usual arguments concerning product forcing, $G(\alpha)$ is $V[G\res\alpha]$-generic for $\bbP_\alpha$ whenever $G\res\alpha+1$ is $V$-generic for $\bbP\res\alpha+1$. Therefore, we have that $\calM_{\alpha+1}\subseteq\calN\subseteq V[G\res\alpha+1]$ are a chain of transitive subclasses, and $V[G\res\alpha]$ is a transitive subclass of $\calN$ as well.

Let $X_\alpha\in\calM_{\alpha+1}$ be the realisation of the name $\ddX$ exhibited in Theorem~\ref{thm:one-step}. Since $\ddX$ is a $\bbP_\alpha$-name it is equivalent to its realisation under $G\res\alpha+1$. Similarly, the realisations of $\dde_\eta$ and $\ddm_\eta$ for appropriate values of $\eta$ are still the desired functions, and so $\calM_{\alpha+1}\vDash``\aleph(X_\alpha)\geq\lambda_\alpha\text{ and }\aleph^*(X_\alpha)\geq\kappa_\alpha"$.

Theorem~\ref{thm:one-step} required no assumptions about the ground model other than $\ZFC$, and so $\calN$, in its role as a symmetric extension of $V[G\res\alpha]$, must satisfy $\aleph(X_\alpha)=\lambda_\alpha$ and $\aleph^*(X_\alpha)=\kappa_\alpha$. Therefore, since any function $f\in\calM_{\alpha+1}$ is also in $\calN$, we cannot have that $f$ is an injection $\lambda_\alpha\to X_\alpha$ or a surjection $X_\alpha\to\kappa_\alpha$. Hence, $\calM_{\alpha+1}$ satisfies $\aleph(X_\alpha)=\lambda_\alpha$ and $\aleph^*(X_\alpha)=\kappa_\alpha$ as well.
\end{proof}

With some care over the construction of $X_\alpha$, we may now prove our main theorem.

\begin{mainthm}
$\ZF$ is equiconsistent with $\ZF+$``for all infinite cardinals $\lambda\leq\kappa$ there is a set $X$ such that $\aleph(X)=\lambda\leq\kappa=\aleph^*(X)"$.
\end{mainthm}

\begin{proof}
Firstly, by Proposition~\ref{prop:prime-directive} and Theorem~\ref{thm:limit-rank}, $\calM\vDash\ZF$. Secondly, by Proposition~\ref{prop:witness-hartlin}, for all $\alpha\in\Ord$, $\calM_{\alpha+1}\vDash``\aleph(X_\alpha)=\lambda_\alpha\text{ and }\aleph^*(X_\alpha)=\kappa_\alpha"$. Note that, for all $\alpha\in\Ord$, the set $X_\alpha$ is constructed as an element of $\power^3(\mu_\alpha)$. Since $\mu_\alpha>\kappa_\alpha\geq\lambda_\alpha$, it must be the case that any function $\lambda_\alpha\to X_\alpha$ or $X_\alpha\to\kappa_\alpha$ must have rank less than $\mu_\alpha^+$. Hence, by Proposition~\ref{prop:prime-directive}, we have that for all $\alpha<\beta$, $\calM_\beta\vDash``\aleph(X_\alpha)=\lambda_\alpha\text{ and }\aleph^*(X_\alpha)=\kappa_\alpha"$. This, combined with Theorem~\ref{thm:limit-rank}, shows that for all $\alpha\in\Ord$, $\calM\vDash``\aleph(X_\alpha)=\lambda_\alpha\text{ and }\aleph^*(X_\alpha)=\kappa_\alpha"$.
\end{proof}

\section{Open questions}

\begin{qn}
What are the limitations of the spectrum of Hartogs--Lindenbaum pairs in arbitrary models of $\ZF$?
\end{qn}
In other words, considering $\{\tup{\lambda,\kappa}\mid\exists X,\alpha(X)=\lambda,\aleph^*(X)=\kappa\}$ in a fixed model of $\ZF$, what are the limitations on this class? Clearly, one requirement is that $\lambda\leq\kappa$. We can also deduce that if this spectrum is exactly $\lambda=\kappa$, then $\AC_\WO$ holds. Is there anything better that we can say about the spectrum or deduce from its properties?

In \cite{ryan-smith_acwo_2024}, the spectrum of models of $\SVC$ are studied and classified into various well-behaved subclasses. For example, if $\calM\vDash\SVC$ then there is a cardinal $\kappa$ such that for all $X$, if $\aleph^*(X)\geq\kappa$ then $\aleph^*(X)$ is a successor cardinal.

\begin{qn}
What type of forcing notions preserve all Hartogs--Lindenbaum values from the ground model?
\end{qn}
Assuming $\ZFC$ holds, the above question has a simple answer: any cardinal preserving forcing will suffice. More generally, in the case of Hartogs's number, any ${<}\Ord$-distributive forcing will not add any injections from an ordinal into a ground model set. So, for example, any $\sigma$-distributive forcing must preserve Dedekind-finiteness. But what about preservations of Lindenbaum numbers as well?

\section{Acknowledgements}

The authors would like to thank Carla Simons for her help and encouragements, as well as Andrew Brooke-Taylor for his comments on an early version of this manuscript. The authors would also like to thank the reviewer for their work in checking this manuscript.

\providecommand{\bysame}{\leavevmode\hbox to3em{\hrulefill}\thinspace}
\providecommand{\MR}{\relax\ifhmode\unskip\space\fi MR }
\providecommand{\MRhref}[2]{%
  \href{http://www.ams.org/mathscinet-getitem?mr=#1}{#2}
}
\providecommand{\href}[2]{#2}

\end{document}